\documentclass{article}
\usepackage{amsthm,amsfonts, amsbsy, amssymb,amsmath,graphicx}
\usepackage{graphics}
\usepackage[english]{babel}
\author{Vassily Olegovich Manturov}
\date{}

\theoremstyle{plain}
\newtheorem{thm}{Theorem}

\theoremstyle{dfn}
\newtheorem{dfn}{Definition}
\newtheorem{rk}{Remark}

\def\R{{\mathbb R}}

 \def\0{{\mathbbf 0}}
 \def\1{{\mathbbf 1}}

 \newcommand{\skcr}{\raisebox{-0.25\height}{\includegraphics[width=0.5cm]{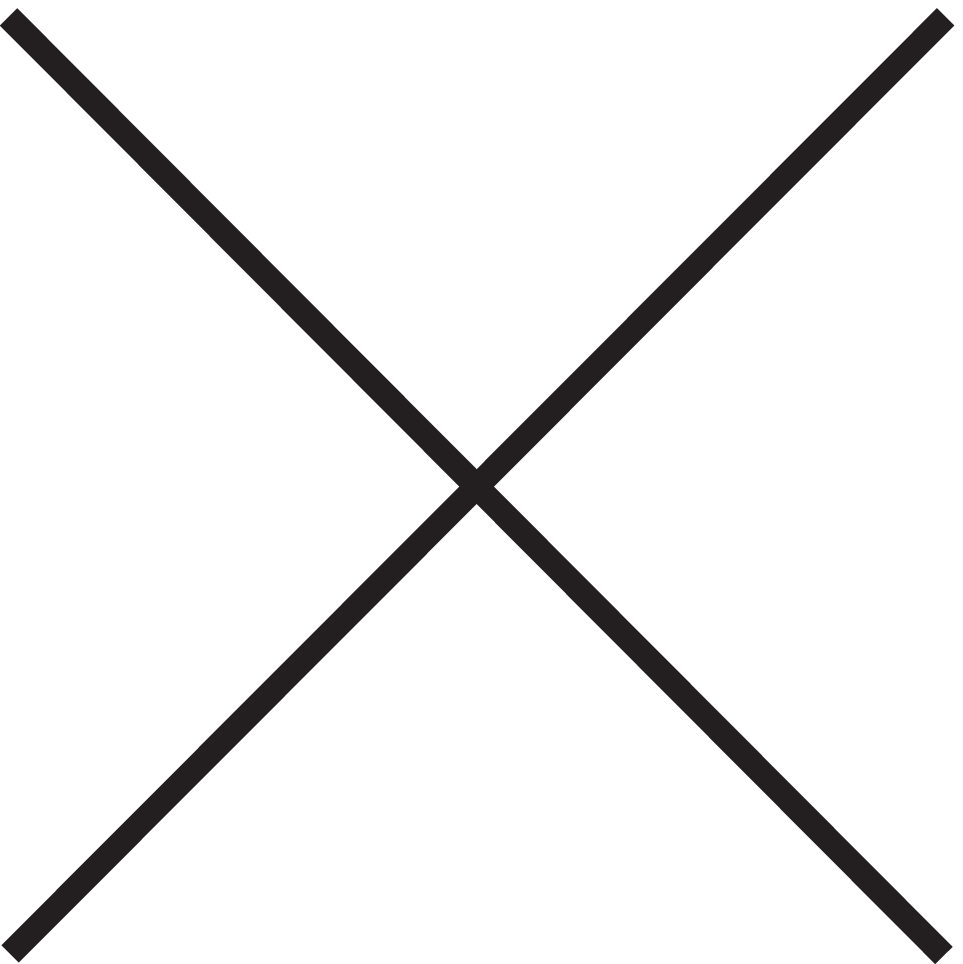}}}
 \newcommand{\skcrv}{\raisebox{-0.25\height}{\includegraphics[width=0.5cm]{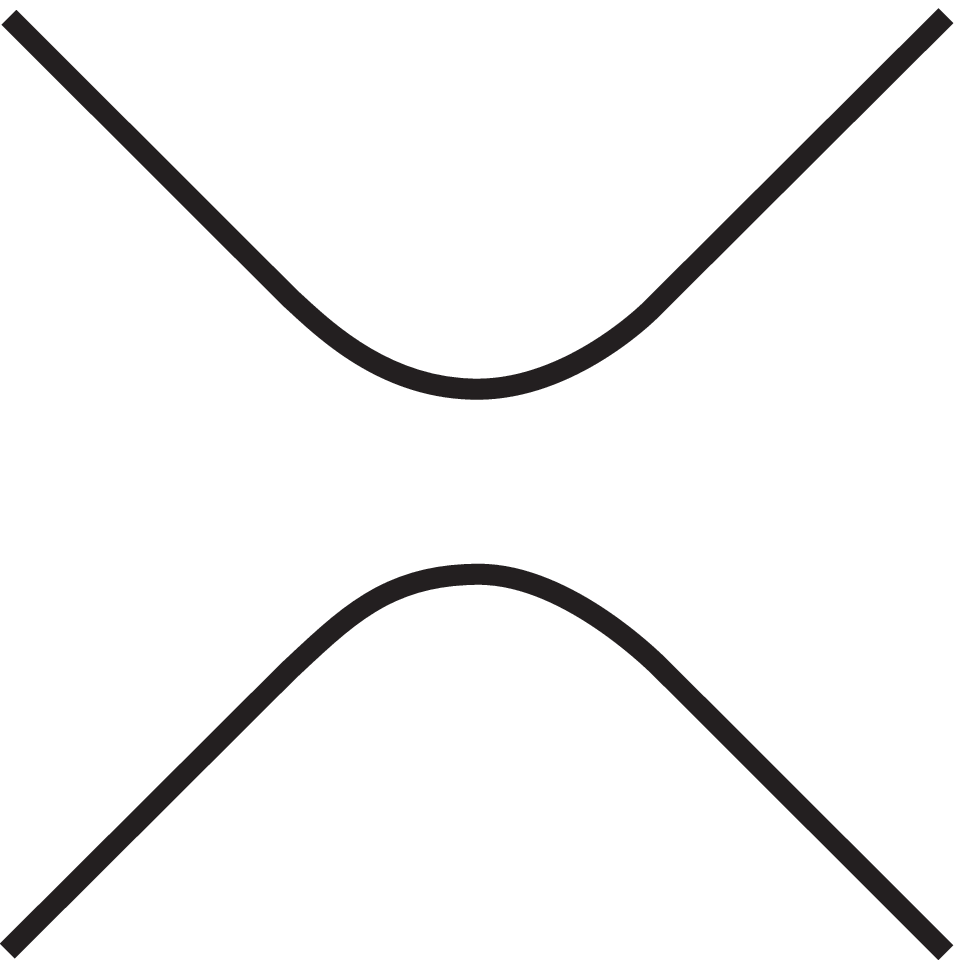}}}
\newcommand{\skcrh}{\raisebox{-0.25\height}{\includegraphics[width=0.5cm]{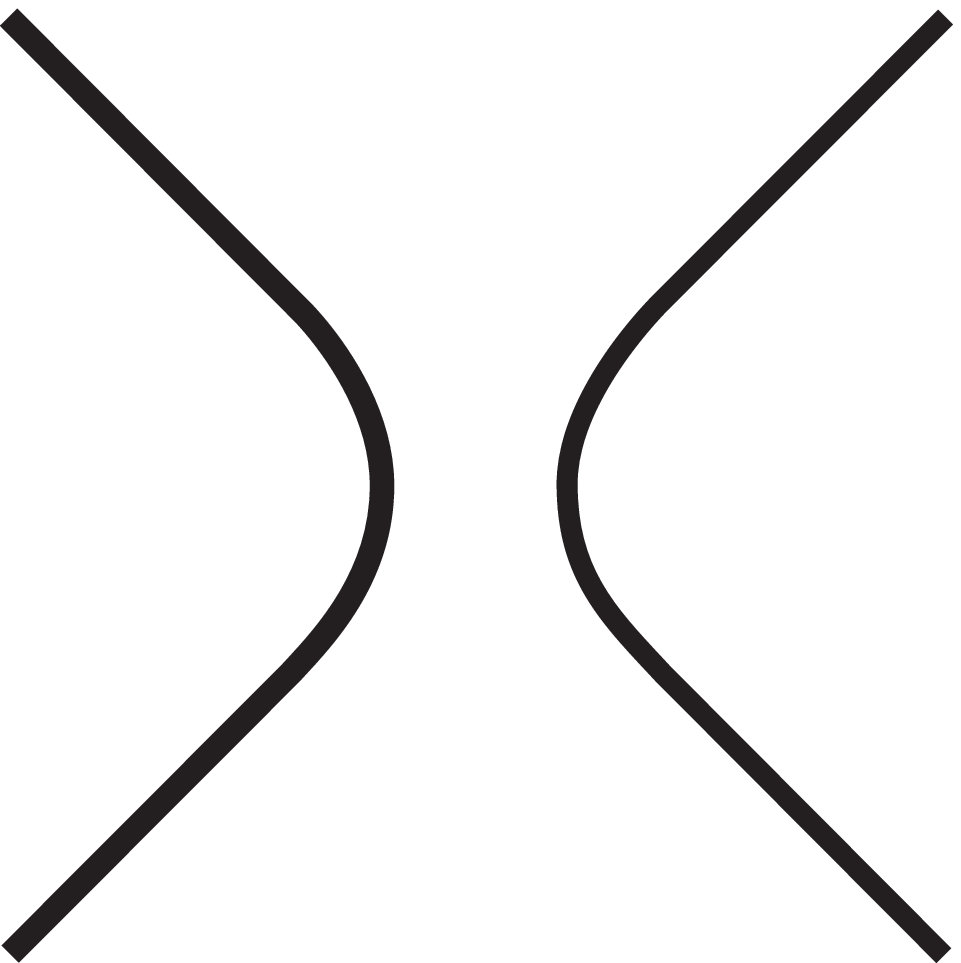}}}

\title{Framed $4$-valent Graph Minor Theory I: Intoduction. A Planarity Criterion and Linkless Embeddability}

\begin{document}

\maketitle

\begin{abstract}
The present paper is the first one in the sequence of papers about a
simple class of  {\em framed $4$-graphs}; the goal of the present
paper is to collect some well-known results on planarity and to
reformulate them in the language of {\em minors}.

The goal of the whole sequence is to prove analogues of the
Robertson-Seymour-Thomas theorems for framed $4$-graphs: namely, we
shall prove that many minor-closed properties are classified by
finitely many excluded graphs.

From many points of view, framed $4$-graphs are easier to consider
than general graphs; on the other hand, framed $4$-graphs are
closely related to many problems in graph theory.

\end{abstract}


Some years ago, a milestone in graph theory was established: as a
result of series of papers by Robertson, Seymour (and later joined
by Thomas) \cite{RS20} proved the celebrated Wagner conjecture
\cite{Wagner} which stated that if a class of graphs (considered up
to homeomorphism) is minor-closed (i.e., it is closed under edge
deletion, edge contraction and isolated node deletion), then it can
be characterized by a finite number of excluded minors. For a
beautiful review of the subject we refer the reader to L.Lov\'asz
\cite{Lovasz}.

This conjecture was motivated by various evidences for concrete
natural minor-closed properties of graphs, such as knotless or
linkless embeddability in $\R^{3}$, planarity or embeddability in a
standardly embedded $S_{g}\subset \R^{3}$.

Here we say that a property $P$ is {\em minor-closed} if for every
graph $X$ possessing this property every minor $Y$ of $G$ possesses
$P$ as well. Later, we shall define the notion of {\em minor} in a
way suitable for framed $4$-graphs.

The most famous evidence of this conjecture is the {\em
Pontrjagin-Kuratowski} planarity criterion which states (in a
slightly different formulation) that a graph is not planar if and
only if it contains one of the two graphs shown in Fig. \ref{PK} as
a minor.

\begin{figure}
\centering\includegraphics[width=200pt]{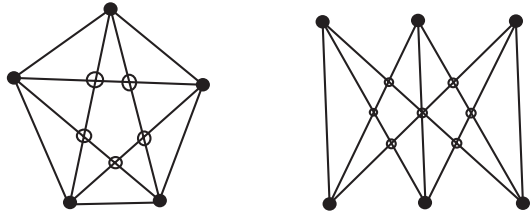} \caption{The two
Kuratowski graphs, $K_{5}$ and $K_{3,3}$} \label{PK}
\end{figure}

\begin{rk}
Throughout the paper (and all subsequent papers in the series), all
graphs are assumed to be finite; loops and multiple edges are
allowed.
\end{rk}

Among all graphs, there is an important class of four-valent  {\em
framed} graphs (or {\em framed} regular $4$-graphs). Here by {\em
framing} we mean a way of indicating which half-edges are opposite
at every vertex. Whenever drawing a framed four-valent graph on the
plain, we shall indicate its vertices by solid dots,
(self)intersection points of edges will be encircled, and the
framing is assumed to be induced from the plane: those half-edges
which are drawn opposite in $\R^{2}$ are thought to be opposite.
Half-edges  of a framed four-valent graph incident to the same
vertex are which are not {\em opposite}, are called {\em adjacent}.

This class of graph is interesting because of its close connection
to classical and virtual knot theory \cite{ManDkld,ManHdlbg},
homotopy classes of curves on surfaces, see also \cite{FM1,FM2}; for
more about virtual knot theory see \cite{MI}.

From time to time we shall admit some broader class of objects than
just framed four-valent graphs. By a {\em $4$-graph} we mean a
finite $1$-complex with every component either being homeomorphic to
a circle or being a graph with all vertices having valency $4$;
components of a $4$-graph homeomorphic to circles will be called
{\em circular components} or {\em circular edges}; by a {\em vertex}
of a framed $4$-graph we mean a vertex of its non-circular
component. By a ({\em non-circular}) edge of a $4$-graph we mean an
edge of its non-circular component. A $4$-graph is {\em framed} if
all non-circular components of it are framed and all circular
components of it are oriented.

There are some natural ways to extend the notion of {\em
minor-closed property} to four-valent framed graphs.

\begin{dfn}
A framed $4$-valent graph $G'$ is a {\em minor} of a framed
$4$-valent graph $G$ if $G'$ can be obtained from $G$ by a sequence
of {\em smoothing} operations ($\skcr\to \skcrv$ and $\skcr\to
\skcrh$) and deletions of connected components.
\end{dfn}

\begin{rk}
Whenever talking about embedding or immersion of a framed $4$-graph
into any $2$-surface we always assume its framing to be preserved:
opposite edges at every crossing should be locally opposite on the
surface.
\end{rk}

\begin{dfn}
We say that a framed $4$-graph $\Gamma$ admits a source-sink
structure if there is an orientation of all edges $\Gamma$ such that
at every vertex of $\Gamma$ some two opposite edges are incoming,
and the other two are emanating. Certainly, for every connected
framed four-valent graph, if a source-sink structure exists, then
there are exactly two such structures.
\end{dfn}

Moreover, it can be easily seen that if $\Gamma$ admits a
source-sink structure then every minor $\Gamma'$ of $\Gamma$ admits
a source-sink structure as well. Indeed, the smoothing operation can
be arranged to preserve the source-sink structure.

So, it is natural to ask many questions about graphs admitting a
source-sink structure.

\begin{rk}
In the present paper, we restrict ourselves to framed $4$-graphs
admitting source-sink structures. Framed $4$-graphs not admitting
source-sink structures will be considered in subsequent papers.
\end{rk}

Denote by $\Delta$ the following framed $4$-graph with $3$ vertices:
it has $3$ vertices $P,Q,R$, and $6$ edges $a,a',b,b',c,c'$ such
that at vertex $P$ the edges $a$ and $a'$ are opposite and both
connect $P$ to $Q$ (in $Q$ they are opposite, as well); $b,b'$
constitute the other pair of opposite edges at $P$; they connect $P$
to $R$, and they are opposite at $R$ as well; finally, $c$ and $c'$
are edges connecting $Q$ and $R$; these edges are opposite both at
$Q$ and at $R$.

\begin{figure}
\centering\includegraphics[width=200pt]{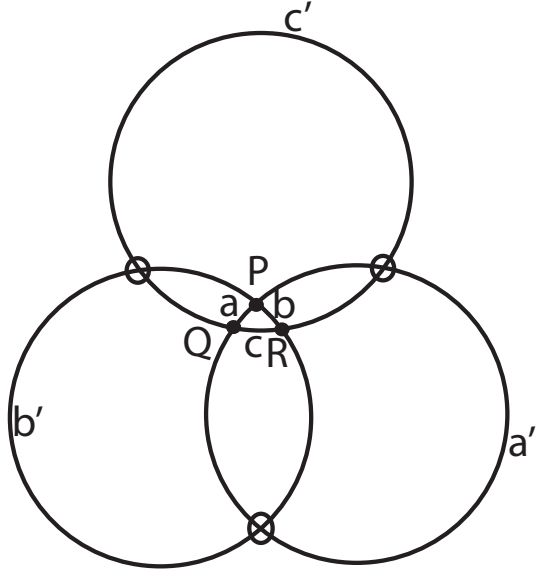} \caption{The
Graph $\Delta$} \label{Delta}
\end{figure}

When drawn immersed in $\R^{2}$, the graph $\Delta$ contains three
pairwise intersecting cycles $(a,a'),(b,b'),(c,c')$; each two of
these three cycles intersect transverselly at one point; thus, an
immersion requires at lease one intersection point for each pair of
these two cycles. In Fig. \ref{Delta} these three immersion points
are encircled.

\begin{dfn}
For a framed $4$-graph $P$ by a {\em loop} we mean either a circular
component (also treated as a map $S^{1}\to P$) or a map $f:S^{1}\to
\Gamma$ which is a bijection everywhere except preimages of
crossings of $\Gamma$.

A  {\em loop} is a circuit if its image is the whole graph $P$
(certainly, only connected framed $4$-graphs admit circuits).

A {\em loop} (resp., circuit) is {\em rotating} if at every crossing
$X$ which has two preimages $Y_{1}$ and $Y_{2}$, the neighbourhoods
of $Y_{1}$ is mapped to two {\em non-opposite} edges.

By abuse of notation, we shall say that a loop (a circuit) passes
through edges if its image contains these edges.
\end{dfn}

\begin{dfn}
Let $L_{1}, L_{2}$ be two loops of a framed $4$-graph $P$; let $X$
be a crossing of $P$; we say that $L_{1}$ and $L_{2}$ intersect
transversely at $X$ if  $L_{1}$ passes through a pair of opposite
edges at $X$ as well $L_{2}$.
\end{dfn}

\begin{dfn}
By a {\em chord diagram} we mean either an oriented circle ({\em
empty}) chord diagram or a cubic graph $D$ consisting of an oriented
cycle (the {\em core}) passing through all vertices of $D$ such that
the complement to it is a disjoint union of edges ({\em chords}) of
the diagram.
\end{dfn}

An easy exercise (see, e.g. \cite{ManVasConj}) shows that {\em every
connected framed $4$-graph admits a rotating circuit}.

Having a circuit $C$ of a framed connected $4$-graph $G$, we define
the chord diagram $D_{C}(G)$, as follows. If $G$ is a circle, then
$D_{C}(G)$ is empty. Think of $C$ as a map $f:S^{1}\to D$; then we
mark by points on $S^{1}$ preimages of vertices of $G$. Thinking of
$S^{1}$ as a core circle and connecting the preimages by chords, we
get the desired cubic graph.

\begin{rk}
Chord diagrams are considered up to combinatorial equivalence.
\end{rk}

\begin{rk}
One can associate chord diagrams not only to {\em rotating
circuits}, but for the present paper we restrict ourselves only with
rotating circuits and framed $4$-graphs admitting a source-sink
structure.
\end{rk}

The opposite operation (of restoring a framed $4$-graph with a
source-sink structure from a chord diagram) is obtained by removing
chords from the chord diagram and approaching two endpoints of each
chord towards each other as shown in Fig. \ref{CDgr}.

\begin{figure}
\centering\includegraphics[width=200pt]{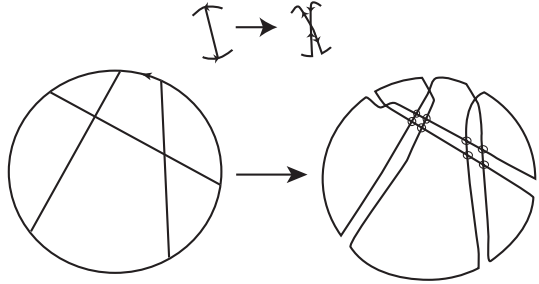} \caption{Restoring
a framed $4$-graph from a chord diagram} \label{CDgr}
\end{figure}

\begin{dfn}
A chord diagram $D'$ is called a {\em subdiagram} of a chord diagram
$D$ if $D$ can be obtained from $D$ by deleting some chords and
their endpoints.
\end{dfn}

It follows from the definition that the removal of a chord from a
chord diagram results in a smoothing of a framed $4$-graph.
Consequently, if $D'$ is a subdiagram of $D$, then the resulting
framed $4$-graph $G(D')$ is a {\em minor} of $G(D)$.

Every embedding $i:P\to \R^{3}$ gives rise to an embedding of every
rotating circuit $C$ of $P$: at each vertex where $C$ touches itself
we perform a smoothing.

We say that two rotating circuits $C_{1},C_{2}$ sharing no edges are
{\em not transverse} if at every vertex which belongs to both
$C_{1}$ and to $C_{2}$ the edges incident to $C_{1}$ are not
opposite at this vertex.

Any embedding of a framed $4$-valent graph in $\R^{3}$ is assumed to
be smooth in the following sense: in the neighbourhood of every
vertex $X$ we require that tangent vectors of opposite half-edges
are opposite. Thus, having a framed $4$-graph $P$ and an embedding
$i:P\to \R^{3}$, we may assume without loss of generality that the
small neighbourhood of every vertex $X$ of $P$ is mapped to a piece
of a $2$-surface containing $X$. Thus, having two rotating loops
$L_{1},L_{2}$ of $P$ with no transverse intersections we can define
the associate the disjoint embedding of $L_{1}$ and $L_{2}$ in
$\R^{3}$ obtained by local smoothing at some vertices. By abusing
notation, we shall talk about {\em images of loops or circuits} in
$\R^{3}$ meaning the cooresponding smoothings (which represent
collection of disjoint curves in $\R^{3}$.

\begin{dfn}
An embedding $i$ of a framed $4$-graph $P$ in $\R^{3}$ with a
source-sink structure is called {\em linkless} if for every two
rotating loops $L_{1}, L_{2}$ without transverse intersection the
linking number of their images is $0$.

Analogously, an embedding $i$ of a framed $4$-graph $P$ in $\R^{3}$
with a source-sink structure is {\em knotless} if the image of the
every rotating loop $L$ is unknotted.
\end{dfn}

This means that in the neighbourhood of such a vertex we can perform
a smoothing of $X$ and an embedding $i$ gives rise to embeddings of
all minors of $P$ defined up to homotopy.

Now we list some {\em minor} properties of framed $4$-valent graphs
(the proof is left for the reader):

\begin{enumerate}

\item Planarity.

\item Existence of an immersion into a fixed surface $\Sigma$ with
no more than $s$ transverse simple intersection points ($s$ fixed).

\item Linkless embeddability (in $\R^{3}$).

\item Knotless embeddability (in $\R^{3}$).

\end{enumerate}

The Main Theorem of the present paper sounds as follows
\begin{thm}
Let $\Gamma$ be a framed $4$-graph admitting a source-sink
structure. Then the following four conditions are equivalent:

\begin{enumerate}

\item Every generic immersion of $\Gamma$ in $\R^{2}$ requires at least $3$
additional crossings;

\item For every embedding of $\Gamma$, there exists a pair of
rotating loops with odd linking number.

\item $\Gamma$ has no linkless embedding in $\R^{3}$;

\item $\Gamma$ is not planar;

\item $\Gamma$ contains $\Delta$ as a minor.

\end{enumerate}

\end{thm}

\begin{proof}

Certainly, 1) yields 4) and 3) yields 4): a planar graph has a
planar {\em embedding} which is an immersion with {\em no additional
points}; moreover, a planar embedding is always linkless.

Our goal is to prove that the non-planarity of a framed $4$-graph
with a source-sink structure yields the existence of $\Delta$ as a
minor. After that, we see that every immersion of $\Delta$ requires
at least $3$ points, which is obvious, and prove that there for
every embedding of  $\Delta$ in $\R^{3}$, there exists a pair of
rotating loops without crossing points having {\em odd linking
number}. The latter automatically means that the embedding is not
linkless.

We follow the proof of Vassiliev's conjectutre \cite{Vas} from
\cite{ManVasConj}. Take a rotating circuit $C$ for $\Gamma$; by
assumption, $\Gamma$ admits a source-sink structure, thus, the chord
diagram $D_{C}(\Gamma)$ contains a $(2n+1)$-gon $\Delta_{2n+1}$ as a
subdiagram, see Fig. \ref{2ngon}.

\begin{figure}
\centering\includegraphics[width=200pt]{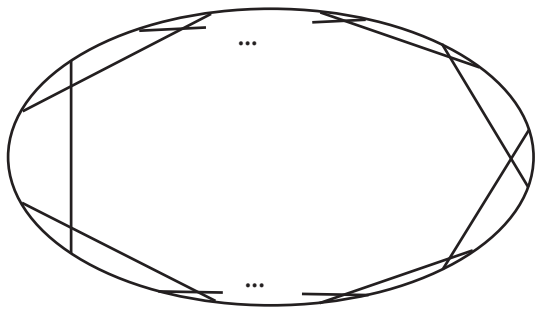} \caption{A
$(2n+1)$-gon} \label{2ngon}
\end{figure}

Consequently, the initial graph will have a minor which corresponds
to the chord diagram $\Delta_{2n+1}$; we denote this framed
$4$-graph by  $Z_{2n+1}$.

Now, we apply the following fact whose prove is left to the reader
as an exercise: $\Delta$ is a minor of $Z_{2n+1}$ for every natural
$n$.

Thus, $\Delta$ is a minor of $\Gamma$, as required.

Let us now prove that there is no linkless embedding of $\Delta$ in
$\R^{3}$; consequently, none exits for $\Gamma$.

Indeed, let us consider the immersion given in Fig. \ref{immer}.

\begin{figure}
\centering\includegraphics[width=200pt]{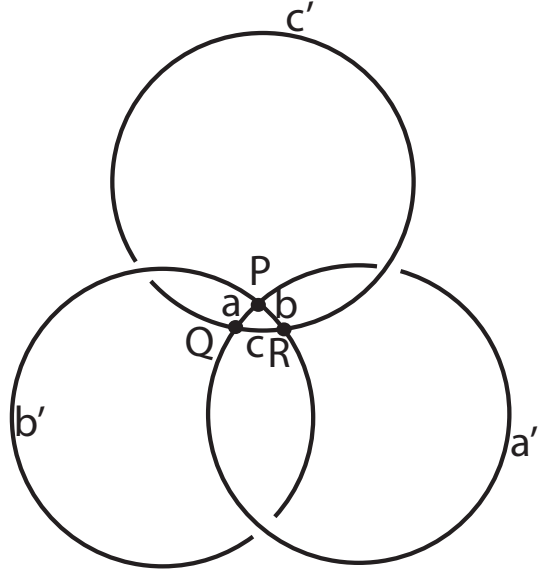} \caption{An
immersion of $\Delta$ in $\R^{3}$} \label{immer}
\end{figure}

Let us consider the following four pairs of cycles $F_{1}=(a,b,c),
F_{2}=({a',b',c'}),
G_{1}=(a,b,c'),G_{2}=(a',b',c),H_{1}=(a,b',c),H_{2}=(a',b,c'),
I_{1}=(a',b,c),I_{2}=(a,b',c')$.

For the immersion given in Fig. \ref{immer} we see that the linking
numbers are $lk(F_{1},F_{2})=0$, whence all linking numbers
$lk(G_{1},G_{2}), lk(H_{1},H_{2}),lk(I_{1},I_{2})$ are congruent to
$1$ modulo $2$.

Thus, the sum of these four linking numbers is odd.

Now, linking numbers do not change under homotopy; thus, this sum
remains odd when applying homotopy to the immersion given in Fig.
\ref{immer}.

Besides homotopy, we can apply some crossing switches in $3$-space.
The whole graph $\Delta$ consists of $6$ edges; if we apply a
crossing switch to an edge with itself (say, $a$ with $a$), it will
make no effect in any of the four summands. Now, if we apply a
crossing switch for an edge with a dash and a corresponding edge
without a prime (say, $a$ and $a'$), this will result in changes
modulo $2$ for all four summands; thus, the total sum will remain
odd.

In the case when we have two letters either both without primes or
both with primes (without loss of generality we may assume they are
$a$ and $b$), two of four summands will remain the same and the
other two will change. Consequently, the parity will remain the
same.

Finally, if we apply a crossing switch to some edges which are not
opposite at some vertex (without loss of generality, we may assume
we deal with $a$ and $b'$), this will change two of four summands:
namely, $lk(F_{1},F_{2})$ and $lk(G_{1},G_{2})$ will change by one.

Thus, the total parity of the sum of linking numbers will not
change.

Thus, we conclude that at least one of these four crossing numbers
will be odd.

\end{proof}

I am grateful to Igor Mikhailovich Nikonov for valuable comments and
to Denis Petrovich Ilyutko for useful discussions.

\end{document}